\documentclass[12pt]{amsart}

\usepackage{amsmath,amssymb,epic,lscape}

\newtheorem{theorem}{Theorem}[section]

\newtheorem{corollary}[theorem]{Corollary}

\theoremstyle{definition}

\theoremstyle{remark}

\numberwithin{equation}{section}

\def\DJ{{\hbox{D\kern-.8em\raise.15ex\hbox{--}\kern.35em}}}
\def\DJo{$\;$\kern-.4em
    \hbox{D\kern-.8em\raise.15ex\hbox{--}\kern.35em okovi\'c}}

\def\vf{{\varphi}}

\def\bZ{{\mbox{\bf Z}}}
\def\bC{{\mbox{\bf C}}}
\def\bR{{\mbox{\bf R}}}

\def\pA{{\mathcal A}}
\def\pN{{\mathcal N}}
\def\pO{{\mathcal O}}
\def\pP{{\mathcal P}}
\def\pS{{\mathcal S}}

\def\tr{{\rm tr\;}}

\def\Pf{{\mbox{\rm Pf\;}}}
\def\pf{{\mbox{\rm pf\;}}}
\def\pl{{\mbox{\rm pl}}}
\def\GL{{\mbox{\rm GL}}}

\def\SO{{\mbox{\rm SO}}}

\def\Ort{{\mbox{\rm O}}}

\renewcommand{\subjclassname}{\textup{2000} Mathematics Subject
Classification}

\begin{document}

\title[Orthogonal invariants of a single matrix]
{On orthogonal and special orthogonal invariants of
a single matrix of small order}

\author[D.\v{Z}. \DJ okovi\'{c}]
{Dragomir \v{Z}. \DJ okovi\'{c}}

\address{Department of Pure Mathematics, University of Waterloo,
Waterloo, Ontario, N2L 3G1, Canada}

\email{djokovic@uwaterloo.ca}

\thanks{
The author was supported by the NSERC Grant A-5285.}


\date{}

\begin{abstract}
Let $\pS_n$ resp. $\pO_n$ be the algebra of polynomial invariants
for the usual conjugation action of $\SO_n(\bC)$ resp. $\Ort_n(\bC)$
on the space $M_n^0$ of traceless $n\times n$ complex matrices.
Note that $\pO_n=\pS_n$ if $n$ is odd. Minimal generating sets of
$\pS_n$ and $\pO_n$ are known for $n\le4$. We construct one for
$\pO_5=\pS_5$. We also construct a Hironaka decomposition of
$\pO_3=\pS_3$ and a new (more economical) one of $\pS_4$.
A simple presentation (with just one syzygy) is obtained for
the algebra $\pS_3$.
\end{abstract}

\maketitle
\subjclassname{ 13A50, 14L35}

\section{Introduction}

Let $M_n=M_n(\bC)$ resp. $M_n^0=M_n^0(\bC)$ denote the space of
all resp.  traceless $n\times n$ matrices over $\bC$,
the field of complex numbers.
We are interested in the conjugation action, $x\to axa^{-1}$,
of the complex special orthogonal group $\SO_n=\SO_n(\bC)$
and the full complex orthogonal group $\Ort_n=\Ort_n(\bC)$ on 
$M_n$ and $M_n^0$. As the identity matrix $I_n$ is fixed under
this action, we shall deal mostly with the space $M_n^0$.

If $n$ is odd, then $-I_n\notin\SO_n$ and, of course, $-I_n$
acts trivially on $M_n$. Hence in that case there is no need
to consider the case of $\Ort_n$.

Let $\pP_n$ be the algebra of polynomial functions
$M_n^0\to\bC$, graded by the degree of the polynomials.
Denote by $\pP_n^d$ the space of homogeneous polynomials
of degree $d$ in $\pP$.
We denote by $\pS_n$ resp. $\pO_n$ the subalgebra of $\pP_n$
which consists of functions
that are invariant under the action of $\SO_n$ resp. $\Ort_n$.
These are graded subalgebras and we denote by $\pS_n^d$
resp. $\pO_n^d$ their homogeneous components of degree $d$.

Let us recall the following classical result on orthogonal
invariants of a single matrix (the First Fundamental
Theorem for the action of $\Ort_n$ on $M_n^0$)
which says that the algebra $\pO_n$
is genarated by the trace functions $\tr w(x,x')$,
where $x\in M_n^0$ is a generic matrix and $x'$ denotes the
transpose of $x$, and $w$ runs through all words in two letters.
The analog of this for $\SO_n$, and $n=2k$ even, has been proved by
Aslaksen, Tan and Zhu \cite{ATZ} not long ago.
In order to obtain a valid set of generators for $\pS_n$
one has to use not only the traces of words $w(x,x')$
but also the polarized pfaffians of $k$-tuples
$w_1(x,x'),\ldots,w_k(x,x')$ of words in $x$ and $x'$.
The definition of polarized pfaffians will be recalled in
Section \ref{Slu_4}. The question of finding a (finite) minimal set of
generators of $\pS_n$ is more subtle and the results are scarce.

The Poincar\'{e} series of $\pS_n$ is defined by
\[ P(M_n^0,\SO_n)=\sum_{d\ge0} \dim(\pS_n^d) t^d. \]
One defines similarly the Poincar\'{e} series $P(M_n^0,\Ort_n)$
of the graded algebra $\pO_n$.
There is a very simple relation between
these Poincar\'{e} series and the analogous ones for the
action on the full matrix space $M_n$:
\begin{eqnarray}
&& P(M_n,\SO_n)=P(M_n^0,\SO_n)/(1-t), \label{SO} \\
&& P(M_n,\Ort_n)=P(M_n^0,\Ort_n)/(1-t). \label{Ort}
\end{eqnarray}
It is known, see \cite[Theorem 6]{BA} and \cite[Theorem 9.1 (a)]{AB},
that these series are in fact rational functions of $t$.
The answer to the following question is apparently not known:
Is there a simple relationship between the Poincar\'{e}
series (\ref{SO}) and (\ref{Ort}) for $n=2k$ even?

The functions $P(M_n,\SO_n)$ have been computed for small values
of $n$ in various places. The case $n=2$ is elementary:
The algebra $\pS_2$ is the polynomial algebra in two generators, 
the pfaffian of $x$ and the trace of $x^2$. Its Poincar\'{e} series is
\[ P(M_2^0,\SO_2)=\frac{1}{(1-t)(1-t^2)}. \]
For the case $n=3$, see e.g. \cite[Lemma 12, (b)]{BA}
and for $n=4$ see \cite[Theorem 9]{WW}.
The list of the Poincar\'{e} series $P(M_n,\SO_n)$
for $n\le6$, with a sample
computation, is given in \cite[Section 7]{DM}.

Even less is known about the functions $P(M_n,\Ort_n)$
when $n=2k$ is even. The case $n=2$ is easy, see e.g.
\cite[Theorem 10.1]{AB} or \cite[\S36]{KS}:
The algebra $\pO_2$ is the polynomial algebra in two generators
$\tr x^2$ and $\tr xx'$. Its Poincar\'{e} series is
\[ P(M_2^0,\Ort_2)=\frac{1}{(1-t^2)^2}. \]
For the case
$n=4$ see \cite[Theorem 2.2]{DM}. Nothing else seems to
be known.

A homogeneous system of parameters (HSOP) of $\pS_n$ is
an algebraically independent set of homogeneous polynomials
$\{f_1,\ldots,f_m\}\subseteq\pS_n$ such that $\pS_n$
is integral over the polynomial algebra
$\bC[f_1,\ldots,f_m]$, which we denote by $\pS_n^\dag$.
As $\SO_n$ is reductive, an HSOP of $\pS_n$ exists (but it
is far from being unique). The cardinality $m$ of an HSOP is
the Krull dimension of $\pS_n$, i.e., $m=n(n+1)/2-1$.
It is known from the general theory
(see \cite{DK}) that $\pS_n$ is free
as a $\pS_n^\dag$-module and has finite rank, say $r$.
We shall refer to the construction of a basis of this free
module as Hironaka's decomposition.
The subalgebra $\pS_n^\dag$ and the rank $r$ depend on
the choice of the HSOP. For instance, in the case $n=4$,
the paper \cite{DM} gives two different HSOP's
for which $r$ is 24 and 32, respectively.

The main objective of this paper is to construct minimal
generating sets (MSG) and Hironaka
decompositions for the algebras $\pS_n$ for some
small values of $n$. Each of the cases $n=3,4,5$ is treated in
a separate section. For $n=3$ we also find a simple presentation
of $\pS_3$. In the case $n=4$ we construct a more ``economical''
HSOP of $\pS_4$ for which the rank $r=16$.
In the case $n=5$ we only construct an MSG of $\pS_5$
(of cardinality 89). 
There is no need to consider separately the algebras $\pO_n$
since $\pO_3=\pS_3$, $\pO_5=\pS_5$ and the algebra $\pO_4$
has bean dealt with in \cite[Theorem 5.1]{DM}.

Although we work over complex numbers, our results are clearly
valid over any algebraically closed field of characteristic 0.

\section{The case $n=3$} \label{Slu_3}

Our objective here is to construct a Hironaka decomposition of
$\pO_3=\pS_3$. Let us start by quoting a theorem of Sibirski\u\i\,
\cite[Theorem 2.45]{KS}

\begin{theorem} \label{Sib-1}
The algebra of polynomial invariants of the compact group
$\Ort(3)$ acting on real $3\times3$ matrices by conjugation
is generated by the traces of the following $7$ matrices
\[  x,\, x^2,\, x^3,\, xx',\, x^2x',\, x^2(x')^2,\,
xx'x^2(x')^2. \]
Moreover this is a minimal set of generators of this algebra.
\end{theorem}

He also shows that one can replace here the generator $\tr x^2(x')^2$
with $\tr (xx')^2$, and $\tr xx'x^2(x')^2$ with $\tr x'x(x')^2x^2$.

The complex version of this theorem, for the complex orthogonal
group $\Ort_3$ acting on $M_3$ by conjugation, is an immediate
consequence. As we prefer to work with $M_3^0$, we should
just omit the first generator, $\tr x$. Hence the following
result is valid:

\begin{theorem} \label{S_3}
The algebra $\pO_3=\pS_3$ is generated by the invariants $E_1,\ldots,E_6$
which are defined as the traces of the matrices:
\[  x^2,\, xx',\, x^3,\, x^2x',\, x^2(x')^2,\, xx'x^2(x')^2, \]
respectively.
Moreover this is a minimal set of generators of $\pS_3$.
\end{theorem}

The Poincar\'{e} series of $\pS_3$ can be written as
\[ P(M_3^0,\SO_3)=\frac{1+t^6}{(1-t^2)^2(1-t^3)^2(1-t^4)}. \]
This expression suggests that there should exist an HSOP
whose degrees are $2,2,3,3,4$. We shall prove that this is
indeed the case.

Let $\pS_3^\dag$ be the subalgebra of $\pS_3$ generated by
the first 5 invariants $E_i$.

\begin{theorem} \label{Hir_3}
The invariants $E_1,\ldots,E_5$ form an HSOP of the algebra
$\pS_3$. The identity element $1$ and $E_6$ form a free basis
of $\pS_3$ as a $\pS_3^\dag$-module.
\end{theorem}

\begin{proof}
It is easy to check that the Jacobian matrix of the invariants
$E_1,\ldots,E_5$ has rank $5$. This implies that these 
invariants are algebraically independent. It is easy to verify
(e.g. by using Maple) that the remaining
generator $E_6$ satisfies the quadratic equation
$E_6^2+c_1E_6+c_2=0$ with coefficients $c_1,c_2\in\pS_3^\dag$
given by
\begin{eqnarray*}
3c_1 &=& E_3^2-3E_2E_5-3E_4^2, \\
144c_2 &=& 72E_1^4E_5+144E_5(E_1E_3E_4+E_5^2)+16E_3^2(E_2^3+E_3^2) \\
 && +24E_1^2E_2E_3(E_3+E_4)-3E_1(E_1^2+E_2^2)(16E_3E_4+3E_1^3) \\
 && +36E_1^2(E_1E_4^2+E_2^2E_5-5E_5^2)+96E_3E_4^2(E_4-E_3) \\
 && +144E_4^2E_5(E_2-E_1)-4E_3^2(E_1^3+24E_2E_5).
\end{eqnarray*}
As $\pS_2=\pS_2^\dag[E_6]$, we deduce that
$\pS_2$ is integral over $\pS_2^\dag$ and, consequently,
$\{E_1,\ldots,E_5\}$ is an HSOP of $\pS_3$.
The second claim now follows easily.
\end{proof}

As a consequence, we obtain the following:

\begin{corollary} \label{pres}
The algebra $\pS_3$ has the following presentation:
\[ \pS_3\cong \bC[t_1,\ldots,t_6]/(t_6^2+c_1t_6+c_2), \]
where the $t_i$'s are $6$ independent commuting variables and
the coefficients $c_1,c_2\in\bC[t_1,\ldots,t_5]$
are obtained from the formulas displayed above by replacing
each $E_i$ with $t_i$ for $i=1,\ldots,5$.
\end{corollary}

\section{The case $n=4$} \label{Slu_4}

Let us write $\pf x$ for the ordinary pfaffian of a
skew-symmetric matrix $x$. We defined in \cite{DM} a version
of the pfaffian for arbitrary matrices $x$ for which we used the
notation $\Pf x$. The definition is very simple:
$\Pf x=\pf(x-x')$. The same definition was already introduced
in \cite{ATZ} where the notation $\tilde{\pf}$ was used instead
of our $\Pf$. Note that if $x$ is skew-symmetric, then we have
$\Pf x=\pf(2x)$.

For the general definition of polarized pfaffians we refer
the reader to \cite{ATZ}. We give the definition only in the case
that we need, the case of $4\times4$ matrices.
Let $s$ and $t$ be two commuting indeterminates.
If $x,y\in M_4$ then their polarized pfaffian,
which we denote by $\pl(x,y)$ as in \cite{ATZ},
is the coefficient of $st$ in the expansion
of the polynomial $\Pf(sx+ty)$. One can easily check that
$\pl(x,y)$ is a bilinear function of $x$ and $y$ and that
$\pl(x,x)=2\Pf x$ is valid for any $4\times4$ matrix $x$.

In Table 1 we provide the list of 20 polynomials
$U_i,V_j\in\pS_4$, where $A$ stands for a generic $4\times4$
matrix of trace 0 and $B$ for its transpose.
Denote by $\pS_4^\dag$ the subalgebra of $\pS_4$ generated
by the first 9 of these generators, the $U_i$'s.

\begin{center}
    \textbf{Table 1:} Generators of $\pS_4$
    $$
    \begin{array}{l c l}
        U_1=\Pf A              &\quad\quad& V_1=\Pf A^2 \\
        U_2=\tr A^2         &     & V_2=\tr A^3B^2 \\
        U_3=\tr AB   &   & V_3=\tr A^2B^2AB \\
        U_4=\tr A^3  &   & V_4=\Pf ABA \\
        U_5=\tr A^2B  &  & V_5=\Pf A^2(A+B) \\
        U_6=\tr A^4 &    & V_6=\tr A^3B^2AB \\
        U_7=\tr (AB)^2 & & V_7=\pl(AB^3,ABA) \\
        U_8=\tr A^2B^2 & & V_8=\tr A^3BA^2B^2 \\
        U_9=\tr A^3BAB & & V_9=\Pf A^2BA \\
		&	 & V_{10}=\tr A^3BA^2BAB \\
		&	 & V_{11}=\pl(BA^2B^2,A^3B)
    \end{array}
    $$
\end{center}

Two different Hironaka decompositions of $\pS_4$ have been constructed
in \cite{DM}, with $r=24$ and $32$ (see the Introduction for the
definition of $r$).  However we failed to observe that there is probably
a better one (with $r=16$) as suggested by the following form of
the Poincar\'{e} series which is taken from \cite{WW}:
\begin{eqnarray*}
P(M_4^0,\SO_4) &=& \frac{1+t^4+t^5+3t^6+2t^7+2t^8+3t^9+t^{10}+t^{11}+t^{15}}
{(1-t^2)^3(1-t^3)^2(1-t^4)^3(1-t^6)} \\
 &=& 1+3t^2+2t^3+10t^4+7t^5+29t^6+25t^7+73t^8 \\
 && +74t^9+172t^{10}+187t^{11}+381t^{12}+431t^{13} \\
 && +785t^{14}+920t^{15}+1539t^{16}+1827t^{17}+\cdots
\end{eqnarray*}
Observe that the degrees $d$ of the nine factors $1-t^d$ in the
above denominator match the degrees of the invariants $U_i$.

We can now construct this new Hironaka decomposition of $\pS_4$.

\begin{theorem} \label{S_4}
(a) The $20$ invariants $U_i$ and $V_j$ generate $\pS_4$
as a complex unital algebra.

(b) The set of $20$ algebra generators listed in (a) is minimal.

(c) The $9$ invariants $U_i$ listed in Table $1$ form a
homogeneous system of parameters of the algebra $\pS_4$.

(d) The algebra $\pS_4$ is a free $\pS_4^\dag$-module of
rank $16$ with basis consisting of the identity element $1$,
the eleven $V_i$'s and the four products $V_1V_2$, $V_2V_3$,
$V_2V_4$ and $V_6V_{10}$.
\end{theorem}

\begin{proof}
(a) Denote by $M$ the $\pS_4^\dag$-module generated
by the 16 invariants listed in part (d) of the theorem.
We have computed the dimensions of the homogeneous components
of $M$ for all degrees $\le17$. These dimensions match the
corresponding coefficients in the Taylor expansion of
$P(M_4^0,\SO_4)$. As $\pS_4$ is generated as a unital algebra by
homogeneous polynomials of degree $\le10$ (see \cite{DM})
this is more than enough to deduce that (a) holds.

(b) Since all minimal generating sets of $\pS_4$ have the same
cardinality, (b) follows from \cite[Theorem 3.1 (c)]{DM}.

(c) It is easy to verify that the Jacobian matrix of the $U_i$'s
has rank 9 and we deduce that the $U_i$'s are algebraically
independent. Hence $\pS_4^\dag$ is a polynomial algebra on
9 generators. Let $J$ be the ideal of $\pP_4$ generated by all
homogeneous polynomials in $\pS_4$ of positive degree. Denote
by $\pN\subseteq M_4^0$ the affine variety defined by $J$.
As $J$ is a graded ideal, $\pN$ is a cone, known as the
Hilbert's null-cone. Denote by $Z$ the affine variety in
$M_4^0$ defined by the equations $U_i=0$ for $1\le i\le9$.
Clearly, we have $\pN\subseteq Z$. Furthermore it is well
known that the $U_i$'s will form an HSOP of $\pS_4$ iff
$\pN=Z$. Our proof of the equality $\pN=Z$ is almost
the same as the proof given in \cite[Theorem 3.1]{DM}.
Hence we can safely omit it.

(d) We know that $\pS_4$ is a free $\pS_4^\dag$-module
of finite rank $r$. From our expression for $P(M_4^0,\SO_4)$
we deduce that $r=16$. The computation mentioned in the
proof of part (a) shows that $M=\pS_4$ and the last assertion of
the theorem follows.
\end{proof}

We conclude this section by a few remarks about our joint paper
\cite{DM} with M.L. MacDonald.

First of all there is a misprint:
In the sixth line of Section 3, $K_1$ should be replaced by $K_7$.
In the same paragraph, we made the observation that the algebra
$\pS_4$ is not generated by the traces and the pfaffians of words
in $x$ and $x'$.
At that time we were not aware of a general result of Aslaksen et al.
\cite{ATZ} from which it follows that $\pS_4$ is generated by
the traces and polarized pfaffians of words in $x$ and $x'$.
For that reason we had to construct some of the generators of
$\pS_4$ by ad hoc methods. These are the generators $J_9$,
$K_7$, $K_9$ and $K_{11}$ in Table 1 of that paper.

Second, we can now give explicit expressions for these four
ad hoc generators in terms of the other
generators $J_i$ and $K_i$ in \cite[Table 1]{DM} and
the generators $U_i$ and $V_i$ defined in Section \ref{Slu_4}.
For $J_9$ and $K_9$
we do not need to use polarized pfaffians as we have
\begin{eqnarray*}
J_9 &=& 2 \left( \Pf A^2(A+B)-\Pf B^2(A+B) \right), \\
4K_9 &=& 4J_9 \left( J_3-2J_2 \right)-8J_1K_5+\Pf BAB^2-\Pf ABA^2
\end{eqnarray*}
where, to be consistent with the notation in \cite{DM},
we write $A$ for $x$ and $B$ for $x'$. Thus $B=A'$.
For the remaining two generators we have more complicated formulas
\begin{eqnarray*}
6K_7 &=& U_1\left( 2U_3U_4+U_2U_4-3U_2U_5 \right)
+2 \left( U_4-3U_5 \right) V_1-12V_7, \\
48K_{11} &=& 2U_1 \left( U_4(12U_8-4U_1^2-13U_6)+3U_5(7U_6+2U_7) \right) \\
 && +4(3U_5+U_4)(2V_4-U_1U_3^2) \\
 && +(3U_5-U_4)\left( 4(U_3V_1+12V_5)-U_1U_2(7U_2+2U_3) \right) \\
 && +24 \left( 2U_2(U_5-U_4)V_1+(2V_1-U_1U_3)V_2 \right. \\
 && \quad \quad \quad \left. +2U_1V_6+(U_3-U_2)V_7-4V_{11} \right),
\end{eqnarray*}
where $V_7$ and $V_{11}$ are defined using polarized pfaffians
(see Table 1).

\section{The case $n=5$} \label{Slu_5}

Our objective in this section is more modest: We shall
construct a minimal set of homogeneous generators (MSG) of
the algebra $\pO_5=\pS_5$. The construction is based on a result
we proved recently in our paper \cite{DZ}. However this result
has to be adapted for our use here, see Theorem \ref{nov-MSG}
in the Appendix.

Let us quote an elementary but useful theorem of Sibirski\u\i\,
\cite[Theorem 4.33]{KS}. In order to state it, it is convenient 
to introduce the linear map
\[ \vf_n:M_n(\bR)\to M_n(\bR)\times M_n(\bR),\quad x\to(x,x'). \]
Let the compact orthogonal group $\Ort(n)$ act on $M_n(\bR)$
by $(a,x)\to axa^{-1}$ and the real general linear group
$\GL_n(\bR)$ act on $M_n(\bR)\times M_n(\bR)$ by the usual
diagonal conjugation action $a\cdot(x,y)=(axa^{-1},aya^{-1})$.
If $f:M_n(\bR)\times M_n(\bR)\to \bR$ is a polynomial function,
then its pullback $\vf_n^*(f)=f\circ\vf$ is a polynomial
function on $M_n(\bR)$. It is obvious that if $f$ is invariant
then so is $\vf_n^*(f)$. Hence, the pullback homomorphism
$\vf_n^*$ maps $\GL_n(\bR)$-invariants to $\Ort(n)$-invariants.

\begin{theorem} \label{Sib-2}
The pullback homomorphism $\vf_n^*$ maps the algebra of
$\GL_n(\bR)$-invariants onto the algebra of $\Ort(n)$-invariants.
\end{theorem}

Clearly, this theorem remains valid in the complex setting where
we replace $M_n(\bR)$ with $M_n$ or $M_n^0$, $\GL_n(\bR)$ with
$\GL_n=\GL_n(\bC)$, and $\Ort(n)$ with $\Ort_n=\Ort_n(\bC)$.

Let $x$ and $y$ denote two generic matrices in $M_5$ or $M_5^0$.
As in our previous paper \cite{DZ}, we denote by $C_{5,2}$ resp.
$C_{5,2}(0)$ the algebra of polynomial invariants of the module
$M_5\times M_5$ resp. $M_5^0\times M_5^0$.
We have recently constructed in \cite[Theorem 5.1]{DZ}
an MSG of the algebra $C_{5,2}(0)$.
This algebra is generated by the traces of all
words $w(x,y)$ in $x$ and $y$. It has an
involution which sends $\tr w(x,y)$ to $\tr w(y,x)$
for each word $w$. The MSG in \cite{DZ} was chosen so that
it is invariant under this involution up to a $\pm$ sign.
More precisely, if $f(x,y)$ is one of our generators then
either $f(y,x)=\pm f(x,y)$ or $f(y,x)=g(x,y)$, where $g(x,y)$
is another generator in our MSG.

But in order to be able to use effectively Theorem \ref{Sib-2}
we need an MSG of $C_{5,2}(0)$ having a different kind of symmetry.
This algebra admits another involution which we denote by
an asterisk ${}^*$. It can be defined on the algebra of
generic matrices, generated over $\bC$ by $x$ and $y$. If
$w(x,y)$ is any word in $x$ and $y$, it sends
this word to the word $w(x,y)^*$ which is obtained from
$w(x,y)$ by interchanging $x$ and $y$ and by
reversing the order of the letters. (The two operations
can be performed in either order, the outcome is the same.)
We need to construct a ${}^*$-invariant MSG of $C_{5,2}(0)$.
This tedious job is carried out in the Appendix
and the result is presented in Table 3.

We can now construct an MSG of $\pS_5$. The Poincar\'{e} series
$P(M_5^0,\SO_5)$ has been computed in \cite{DM}. It is a rational
function in $t$ which can be written (not in lowest terms)
with numerator the palindromic polynomial
\begin{eqnarray*}
&& 1+t^5+3t^6+4t^7+8t^8+8t^9+15t^{10}+15t^{11}+24t^{12} \\
&& +26t^{13}+34t^{14}+41t^{15}+46t^{16}+50t^{17}+52t^{18}+52t^{19} \\
&& +50t^{20}+\cdots+3t^{31}+t^{32}+t^{37}
\end{eqnarray*}
and denominator $(1-t^2)^2(1-t^3)^2(1-t^4)^4(1-t^5)^3(1-t^8)$.
This suggests that $\pS_5$ should have an HSOP whose degrees are
2,2,3,3,4,4,4,4,5,5,5, 6,6,8. On the basis of our
computations, we conjecture that the traces of the following 14
words in $x$ and $y=x'$ form such an HSOP:
\begin{eqnarray*}
&& x^2,\, xy; \quad x^3,\, x^2y; \quad x^4,\, x^3y,\, x^2y^2,\, (xy)^2; \\
&& x^4y,\, x^3y^2,\, x^2yxy; \quad x^4y^2,\, x^3y^3; \quad x^4y^2xy.
\end{eqnarray*}

\begin{center}
\bf{Table 2: An MSG of the algebra $\pS_5$}
\begin{eqnarray*}
&& x^2,\, xy; \quad x^3,\, x^2y; \quad x^4,\, x^3y,\, x^2y^2,\, (xy)^2; \\
&& x^5,\, x^4y,\, x^3y^2,\, x^2yxy; \quad x^4y^2,\, (x^2y)^2,\, x^3y^3,\, 
x^2y^2xy,\, y^2x^2yx; \\ 
&& x^4yxy,\, x^4y^3,\, x^3y^2xy,\, x^3yxy^2; \quad 
x^4y^2xy,\, x^4yxy^2,\, x^4yx^2y, \\
&& x^3y^2x^2y,\, x^4y^4,\, x^3y^3xy,\, y^3x^3yx,\, 
x^2y^2(xy)^2,\, y^2x^2(yx)^2; \\
&& x^4y^2x^2y,\, y^4x^2y^2x,\, x^4y^2xy^2,\, 
x^3y^2(xy)^2,\, x^3(yx)^2y^2,\, x^3y^3x^2y, \\
&& x^3yxyx^2y,\, x^2y^2xyx^2y; \quad x^4y^3xy^2,\, x^4yxy^2xy,\, x^4y^2x^2y^2, \\
&& x^4y^2(xy)^2,\, x^3y^2x^2yxy,\, x^3yx^2yxy^2,\, x^4yx^2yxy, \\
&& x^4y^2x^3y,\,  x^4y^4xy,\, y^4x^4yx,\, x^3y^3(xy)^2,\, y^3x^3(yx)^2, \\
&& x^3y^3x^2y^2,\, x^2yxy^2(xy)^2; \quad x^4y^4x^2y,\, x^4yx^2y^4,\, x^4y^3x^3y, \\
&& x^4y^3x^2y^2,\, x^4y^3(xy)^2,\, x^4y^2x^2yxy,\, x^4y^2xyxy^2,\, x^4yx^3yxy, \\
&& x^4yx^2y^2xy,\, x^4yx^2yxy^2,\, x^3y^2x^2y^2xy,\, x^3(yx)^2y^2xy; \\
&& x^4y^4x^3y,\, x^4y^3x^3y^2,\, x^4y^3x^2yxy,\, x^4y^3xy^2xy, \\
&& x^4y^2x^3yxy,\, x^4y^2(x^2y)^2,\, x^4y^2xy^2x^2y,\, x^4y^2xyx^2y^2, \\
&& x^4yx^3yx^2y,\, x^4yx^3yxy^2,\, x^4(yx)^2y^2xy,\, x^4y^2x^2y^4, \\
&& x^3y^3x^2yxy^2,\,  x^3y^2(xy)^2yxy; \quad x^4y^3x^3yxy,\, x^4y^3x^2y^2xy, \\ 
&& x^4y^3x^2yx^2y,\, x^4y^2x^3yx^2y,\, x^4(y^2x^2)^2y,\,
x^4y^2(xy)^2yxy,\, x^4y^3x^2y^4; \\
&& x^4y^4x^3yxy,\, x^4y^3x^3y^2xy,\, x^4y^3x^2yx^3y; \quad x^2yxy^3x^2yxy^2xy.
\end{eqnarray*}
\end{center}

The Taylor expansion gives
\begin{eqnarray*}
P(M_5^0,\SO_5) &=& 1+2t^2+2t^3+7t^4+8t^5+20t^6+26t^7+60t^8+82t^9 \\
&& +164t^{10}+236t^{11}+437t^{12}+640t^{13}+1104t^{14} \\
&& +1634t^{15}+2678t^{16}+3940t^{17}+6206t^{18}+\cdots
\end{eqnarray*}

\begin{theorem} \label{S_5}
An MSG of $\pS_5$ has cardinality $89$. Such an MSG, consisting of
homogeneous polynomials, is given by the traces of the $89$ words
$w(x,y)$ listed in Table $2$. The matrix $y$ is used in that table
as an abbreviation for the transpose $x'$ of $x$.
\end{theorem}

\begin{proof}
By setting $y=x'$ in the MSG for $C_{5,2}(0)$ we obtain a set, $S$, of
generators of $\pS_5$ (see Theorem \ref{Sib-2}).
After setting $y=x'$, two words that occur inside a pair of braces 
in Table 3 will have the same trace. Since there are 73 such pairs,
we have $|S|=73+25=98$. By checking for minimality, we find that
the following 9 words are redundant:
\begin{eqnarray*}
&& x^4y^4x^2y^2,\, y^3x^3y^2xyx^2,\, x^4yxy^2x^2y^3,\, x^4yxy^2xyxy^2,\, 
x^4y^4x^3y^3, \\
&& x^3y^2xy^2x^2yxy^2,\, x^4y^2x^2yxy^4,\, x^4yx^2y^2x^3y^3,\, x^4y^2x^3yx^2yxy.
\end{eqnarray*}
Hence the traces of the 89 words listed in Table 2 form an MSG of $\pS_5$.
\end{proof}

\section{Appendix: A new MSG of $C_{5,2}(0)$} \label{app}

The algebra $C_{5,2}(0)$ is not only $\bZ$-graded but also
$\bZ^2$-graded. Indeed we can assign degree $(1,0)$ to each
coordinate function of $x$ and degree $(0,1)$ to those of $y$.
The MSG constructed in \cite{DZ} is $\bZ^2$-graded, i.e.,
each of its elements is homogeneous with respect to this
$\bZ^2$-gradation. The new MSG that we are going to construct
also has this property. Moreover the number of elements of
each bidegree in the two MSG's are the same. These facts
were useful in the actual construction of this new MSG.
The $\bZ^2$-graded Poincar\'{e} series of $C_{5,2}(0)$
plays an important role in our computations.
It has been computed in our paper \cite[Theorem 2.2]{DZ}. 
Its Taylor expansion up to and including
the terms of order 15 is:

\begin{center}
\begin{eqnarray*}
&& P(M_5^0\times M_5^0,\GL_5) = 1+{s}^{2}+st+{t}^{2}+{s}^{3}+{s}^{2}t
+s{t}^{2}+{t}^{3}+2\,{s}^{4} \\
&& +2\,{s}^{3}t+4\,{s}^{2}{t}^{2}+2\,s{t}^{3}+2\,{t}^{4}
+2\,{s}^{5}+3\,{s}^{4}t+5\,{s}^{3}{t}^{2}+5\,{s}^{2}{t}^{3} \\
&& +3\,s{t}^{4}+2\,{t}^{5}+3\,{s}^{6}+4\,{s}^{5}t+10\,{s}^{4}{t}^{2}
+11\,{s}^{3}{t}^{3}+10\,{s}^{2}{t}^{4}+4\,s{t}^{5} \\
&& +3\,{t}^{6}+3\,{s}^{7}+6\,{s}^{6}t+13\,{s}^{5}{t}^{2}
+18\,{s}^{4}{t}^{3}+18\,{s}^{3}{t}^{4}+13\,{s}^{2}{t}^{5} \\
&& +6\,s{t}^{6}+3\,{t}^{7}+5\,{s}^{8}+8\,{s}^{7}t+21\,{s}^{6}{t}^{2}
+30\,{s}^{5}{t}^{3}+40\,{s}^{4}{t}^{4}+30\,{s}^{3}{t}^{5} \\
&& +21\,{s}^{2}{t}^{6}+8\,s{t}^{7}+5\,{t}^{8}+5\,{s}^{9}+10\,{s}^{8}t
+26\,{s}^{7}{t}^{2}+46\,{s}^{6}{t}^{3} \\
&& +61\,{s}^{5}{t}^{4}+61\,{s}^{4}{t}^{5}+46\,{s}^{3}{t}^{6}
+26\,{s}^{2}{t}^{7}+10\,s{t}^{8}+5\,{t}^{9}+7\,{s}^{10} \\
&& +13\,{s}^{9}t+38\,{s}^{8}{t}^{2}+66\,{s}^{7}{t}^{3}
+105\,{s}^{6}{t}^{4}+115\,{s}^{5}{t}^{5}+105\,{s}^{4}{t}^{6} \\
&& +66\,{s}^{3}{t}^{7}+38\,{s}^{2}{t}^{8}+13\,s{t}^{9}+7\,{t}^{10}
+7\,{s}^{11}+16\,{s}^{10}t+46\,{s}^{9}{t}^{2} \\
&& +92\,{s}^{8}{t}^{3}+152\,{s}^{7}{t}^{4}+193\,{s}^{6}{t}^{5}
+193\,{s}^{5}{t}^{6}+152\,{s}^{4}{t}^{7}+92\,{s}^{3}{t}^{8} \\
&& +46\,{s}^{2}{t}^{9}+16\,s{t}^{10}+7\,{t}^{11}+10\,{s}^{12}
+20\,{s}^{11}t+62\,{s}^{10}{t}^{2}+125\,{s}^{9}{t}^{3} \\
&& +229\,{s}^{8}{t}^{4}+310\,{s}^{7}{t}^{5}+362\,{s}^{6}{t}^{6}
+310\,{s}^{5}{t}^{7}+229\,{s}^{4}{t}^{8}+125\,{s}^{3}{t}^{9} \\
&& +62\,{s}^{2}{t}^{10}+20\,s{t}^{11}+10\,{t}^{12}+10\,{s}^{13}
+24\,{s}^{12}t+74\,{s}^{11}{t}^{2}+163\,{s}^{10}{t}^{3} \\
&& +309\,{s}^{9}{t}^{4}+470\,{s}^{8}{t}^{5}+584\,{s}^{7}{t}^{6}
+584\,{s}^{6}{t}^{7}+470\,{s}^{5}{t}^{8}+309\,{s}^{4}{t}^{9} \\
&& +163\,{s}^{3}{t}^{10}+74\,{s}^{2}{t}^{11}+24\,s{t}^{12}
+10\,{t}^{13}+13\,{s}^{14}+29\,{s}^{13}t+95\,{s}^{12}{t}^{2} \\
&& +210\,{s}^{11}{t}^{3}+429\,{s}^{10}{t}^{4}+683\,{s}^{9}{t}^{5}
+941\,{s}^{8}{t}^{6}+1024\,{s}^{7}{t}^{7}+941\,{s}^{6}{t}^{8} \\
&& +683\,{s}^{5}{t}^{9}+429\,{s}^{4}{t}^{10}+210\,{s}^{3}{t}^{11}
+95\,{s}^{2}{t}^{12}+29\,s{t}^{13}+13\,{t}^{14} \\
&& +14\,{s}^{15}+34\,{s}^{14}t+111\,{s}^{13}{t}^{2}
+265\,{s}^{12}{t}^{3}+553\,{s}^{11}{t}^{4}+956\,{s}^{10}{t}^{5} \\
&& +1396\,{s}^{9}{t}^{6}+1681\,{s}^{8}{t}^{7}+1681\,{s}^{7}{t}^{8}
+1396\,{s}^{6}{t}^{9}+956\,{s}^{5}{t}^{10} \\
&& +553\,{s}^{4}{t}^{11}+265\,{s}^{3}{t}^{12}+111\,{s}^{2}{t}^{13}
+34\,s{t}^{14}+14\,{t}^{15}+\cdots
\end{eqnarray*}
\end{center}

\newpage
\begin{center}
\bf{Table 3: A ${}^*$-invariant MSG of the algebra $C_{5,2}(0)$}
\begin{eqnarray*}
&& \{x^2,y^2\},\, xy; \\
&& \{x^3,y^3\},\, \{x^2y,xy^2\}; \\
&& \{x^4,y^4\},\, \{x^3y,xy^3\},\, x^2y^2,\, (xy)^2; \\
&& \{x^5,y^5\},\, \{x^4y,xy^4\},\, \{x^3y^2,x^2y^3\},\, \{x^2yxy,xyxy^2\}; \\
&& \{x^4y^2,x^2y^4\},\, \{(x^2y)^2,(xy^2)^2\}\},\, x^3y^3,\, 
x^2y^2xy,\, y^2x^2yx; \\ 
&& \{x^4yxy,xyxy^4\},\, \{x^4y^3,x^3y^4\},\, \{x^3y^2xy,xyx^2y^3\},\, 
\{x^3yxy^2,x^2yxy^3\}; \\
&& \{x^4y^2xy,xyx^2y^4\},\, \{x^4yxy^2,x^2yxy^4\},\, \{x^4yx^2y,xy^2xy^4\},\, \\
&& \{x^3y^2x^2y,xy^2x^2y^3\},\, x^4y^4,\, x^3y^3xy,\, y^3x^3yx,\, 
x^2y^2(xy)^2,\, y^2x^2(yx)^2; \\
&& \{x^4y^2x^2y,xy^2x^2y^4\},\, \{y^4x^2y^2x,yx^2y^2x^4\},\,
\{x^4y^2xy^2,x^2yx^2y^4\},\, \\
&& \{x^3y^2(xy)^2,(xy)^2x^2y^3\},\, \{x^3(yx)^2y^2,x^2(yx)^2y^3\},\, 
\{x^3y^3x^2y,xy^2x^3y^3\},\, \\
&& \{x^3yxyx^2y,xy^2xyxy^3\},\, \{x^2y^2xyx^2y,xy^2xyx^2y^2\}; \\
&& \{x^4y^3xy^2,x^2yx^3y^4\},\, \{x^4yxy^2xy,xyx^2yxy^4\},\, 
\{x^4y^2x^2y^2,x^2y^2x^2y^4\},\, \\
&& \{x^4y^2(xy)^2,(xy)^2x^2y^4\},\, \{x^3y^2x^2yxy,xyxy^2x^2y^3\},\, \\ 
&& \{x^3yx^2yxy^2,x^2yxy^2xy^3\},\, \{x^4yx^2yxy,xyxy^2xy^4\},\, \\
&& \{x^4y^2x^3y,xy^3x^2y^4\},\,  x^4y^4xy,\, y^4x^4yx,\, 
x^3y^3(xy)^2,\, y^3x^3(yx)^2,\, \\
&& x^3y^3x^2y^2,\, x^2yxy^2(xy)^2; \\
&& \{x^4y^4x^2y,xy^2x^4y^4\},\, \{x^4yx^2y^4,x^4y^2xy^4\},\, 
\{x^4y^3x^3y,xy^3x^3y^4\},\, \\
&& \{x^4y^3x^2y^2,x^2y^2x^3y^4\},\, \{x^4y^3(xy)^2,(xy)^2x^3y^4\},\, \\ 
&& \{x^4y^2x^2yxy,xyxy^2x^2y^4\},\, \{x^4y^2xyxy^2,x^2yxyx^2y^4\},\, \\
&& \{x^4yx^3yxy,xyxy^3xy^4\},\, \{x^4yx^2y^2xy,xyx^2y^2xy^4\},\, \\
&& \{x^4yx^2yxy^2,x^2yxy^2xy^4\},\, \{x^3y^2x^2y^2xy,xyx^2y^2x^2y^3\},\, \\ 
&& \{x^3(yx)^2y^2xy,xyx^2(yx)^2y^3\}; \\
&& \{x^4y^4x^3y,xy^3x^4y^4\},\, \{x^4y^3x^3y^2,x^2y^3x^3y^4\},\, \\
&& \{x^4y^3x^2yxy,xyxy^2x^3y^4\},\, \{x^4y^3xy^2xy,xyx^2yx^3y^4\},\, \\
&& \{x^4y^2x^3yxy,xyxy^3x^2y^4\},\, \{x^4y^2(x^2y)^2,(xy^2)^2x^2y^4\},\, \\
&& \{x^4y^2xy^2x^2y,xy^2x^2yx^2y^4\},\, \{x^4y^2xyx^2y^2,x^2y^2xyx^2y^4\},\, \\
&& \{x^4yx^3yx^2y,xy^2xy^3xy^4\},\, \{x^4yx^3yxy^2,x^2yxy^3xy^4\},\, \\
&& \{x^4(yx)^2y^2xy,xyx^2(yx)^2y^4\},\, x^4y^4x^2y^2,\, x^4y^2x^2y^4,\, \\
&& x^3y^3x^2yxy^2,\, y^3x^3y^2xyx^2,\, x^3y^2(xy)^2yxy+xyx(xy)^2x^2y^3; \\
&& \{x^4y^3x^3yxy,xyxy^3x^3y^4\},\, \{x^4y^3x^2y^2xy,xyx^2y^2x^3y^4\},\, \\ 
&& \{x^4y^3x^2yx^2y,xy^2xy^2x^3y^4\},\, \{x^4y^2x^3yx^2y,xy^2xy^3x^2y^4\},\,  \\
&& \{x^4(y^2x^2)^2y,x(y^2x^2)^2y^4\},\, 
\{x^4y^2(xy)^2yxy,xyx(xy)^2x^2y^4\},\, \\
\end{eqnarray*}

\bf{Table 3 (continued)}
\begin{eqnarray*}
&& \{x^4y^3x^2y^4,x^4y^2x^3y^4\},\, \{x^4yxy^2x^2y^3,x^3y^2x^2yxy^4\},\, \\ 
&& \{x^4yxy^2xyxy^2,x^2yxyx^2yxy^4\}; \\
&& \{x^4y^4x^3yxy,xyxy^3x^4y^4\},\, \{x^4y^3x^3y^2xy,xyx^2y^3x^3y^4\},\, \\ 
&& \{x^4y^3x^2yx^3y,xy^3xy^2x^3y^4\},\, x^4y^4x^3y^3,\, \\
&& x^3y^2xy^2x^2yxy^2+x^2yxy^2x^2yx^2y^3,\, x^4y^2x^2yxy^4+x^4yxy^2x^2y^4; \\
&& \{x^2yxy^3x^2yxy^2xy,xyx^2yxy^2x^3yxy^2\},\, 
\{x^4yx^2y^2x^3y^3,x^3y^3x^2y^2xy^4\},\, \\
&& \{x^4y^2x^3yx^2yxy,xyxy^2xy^3x^2y^4\}.
\end{eqnarray*}
\end{center}

Recall the involution ${}^*$ defined in the previous section.
\begin{theorem} \label{nov-MSG}
The algebra $C_{5,2}(0)$ of polynomial $\GL_5$-invariants of the module
$M_5^0\times M_5^0$ admits an MSG, consisting of bihomogeneous
polynomials, which is invariant under the involution ${}^*$.
Such an MSG is given by the traces of the $171$ matrices
listed in Table $3$.
(Each pair of matrices singled out by the braces has the form $\{w,w^*\}$
where $w$ is a word in $x$ and $y$.)
\end{theorem}

\begin{proof}
The proof is computational (we used Maple). The straightforward algorithm
is described for instance in \cite[Section 2.6]{DK}. The main point is to 
verify that the dimensions of the $\bZ^2$-homogeneous components of the
$\bZ^2$-graded algebra $\pA$ generated by the above 171 traces match the
coefficients of the $\bZ^2$-graded Poincar\'{e} series of the algebra
$C_{5,2}(0)$. 
We have verified that that this is indeed so, i.e., each of the
coefficients of $s^i t^j$ with $i+j\le 15$, in the above Taylor expansion,
is equal to the dimension of the $(i,j)$-th homogeneous component of $\pA$.
As we know from \cite{DZ} that $C_{5,2}(0)$ is
generated by polynomials of degree $\le15$, the first assertion follows.

The second assertion is easy to verify by inspection of Table 3.
We only point out that each matrix listed in this table,
which is not inside the braces, has the property that its trace
remains unchanged when we apply the involution ${}^*$.
For instance, let us take the matrix $y^3x^3y^2xyx^2$ whose
trace has bidegree $(6,6)$. By applying the involution ${}^*$
we obtain the matrix $y^2xyx^2y^3x^3$. Now observe that this
matrix can be obtained from the original one by a cyclic
permutation of its factors. Consequently these two matrices
indeed have the same trace.
\end{proof}

Let us remark that all but three of the matrices listed in Table 3
are just words in $x$ and $y$.
In the three exceptional cases they have the
form $w+w^*$ where $w$ is a word in $x$ and $y$.

\end{document}